\documentclass{amsart}

\usepackage{amsmath, amssymb, amsbsy}
\usepackage{array}
\usepackage{graphpap, color, paralist, pstricks}
\usepackage[mathscr]{eucal}
\usepackage[pdftex]{graphicx}
\usepackage[pdftex,colorlinks,backref=page,citecolor=blue]{hyperref}
\usepackage{pifont}

\usepackage{mathtools, thmtools}
\usepackage{cleveref}
\usepackage{float}
\usepackage{crossreftools}

\newcounter{bullet}

\usepackage{setspace}
\setdisplayskipstretch{1}

\usepackage{geometry}
\usepackage{comment}
\newtheorem{thm}{Theorem}[section]

\newtheorem{cor}[thm]{Corollary}

\newtheorem{conj}[thm]{Conjecture}

\theoremstyle{definition}

\newtheorem{claim}[thm]{Claim}

\newtheorem{remark}[thm]{Remark}

\crefname{lem}{lemma}{lemmas}
\crefname{thm}{theorem}{theorems}

\newcommand{\gl}{\lambda}

\newcommand{\RR}{\mathbb{R}}

\newcommand{\cA}{\mathcal{A} }
\newcommand{\cB}{\mathcal{B} }

\newcommand{\cF}{\mathcal{F} }
\newcommand{\cG}{\mathcal{G} }

\newcommand{\beq}[1]{\begin{equation}\label{#1}}
\newcommand{\enq}[0]{\end{equation}}

\newcommand{\eps}{\epsilon}

\newcommand{\nin}[0]{\noindent}

\newcommand{\sub}[0]{\subseteq}

\newcommand{\0}[0]{\emptyset}
\newcommand{\ra}[0]{\rightarrow}

\newcommand{\pr}[0]{\mathbb{P}}

\renewcommand{\eta}{\left(\left(\frac{q}{2}\right)^2\right)}

\newenvironment{subproof}[1][\proofname]{
  
  \begin{proof}[#1]
}{
  \end{proof}
}

\begin{document}

\title{A dimension-free comparison between expectation thresholds and fractional expectation thresholds}

\author[J. Park]{Jinyoung Park}
\address{Department of Mathematics, The Courant Institute School of Mathematics, Computing, and Data Science, New York University}
\email{jinyoungpark@nyu.edu}

\begin{abstract}
We prove a dimension-free comparison between the expectation threshold
$q(\cF)$ and the fractional expectation threshold $q_f(\cF)$ for any
nontrivial increasing family $\cF$ on a finite ground set. Specifically, we show that
there is a universal constant $K>0$ such that
\[
q_f(\cF)\le Kq(\cF)\max\{1,\log\log(1/q(\cF))\}.
\]
Combining this comparison with a recent result of Li [arXiv:2609.08967] on the fractional version of Talagrand's discrete Convexity Conjecture, we obtain a dimension-independent bound toward the conjecture.
\end{abstract}

\maketitle

\section{Introduction}

In this note, we show a dimension-free comparison between the expectation threshold and the fractional expectation threshold (\Cref{thm:MTez} below). The comparison follows rather quickly from \cite[Theorem 1.3]{pham2025sharp}, a quantitative strengthening of \cite[Theorem 1.5]{park2024conjecture}.

We first briefly recall the necessary definitions. (For a more detailed introduction to the background, see, e.g., \cite{demarco2015note,frankston2022problem}.) All logarithms are natural, unless a base is explicitly indicated.

We use $V$ for a finite set, and write $2^V$ for its power set. A family $\cF \sub 2^V$ is called \textit{increasing} if $A \in \cF$ and $B \supseteq A$ imply $B \in \cF$. For $\cG \sub 2^V$ we use $\langle \cG \rangle$ for the increasing family generated by $\cG$, namely $\{B \sub V: \exists A \in \cG, B \supseteq A\}.$ 
Throughout, we assume that  $\cF \sub 2^V$ is increasing and nontrivial, i.e., $\cF \notin \{\emptyset, 2^V\}$.

We call $\cF$  \textit{$p$-small} if there exists a $\cG \sub 2^V$ such that
\[\langle \cG \rangle \supseteq \cF \text{ and } \sum_{S \in \cG} p^{|S|} \le 1/2,\]
and the \textit{expectation threshold} of $\cF$ is $q(\cF)=\max\{p:\text{$\cF$ is $p$-small}\}$. Say $\cF$ is \textit{weakly $p$-small} if there is a function $\gl:2^V \ra \RR^+ (:= [0, \infty))$ such that
\beq{eq:wps}
\sum_{S \sub I} \gl_S \ge 1 \,\, \forall I \in \cF \,\,
\text{ and } \,\,\sum_S \gl_S p^{|S|} \le 1/2.
\enq
We then define the \textit{fractional expectation threshold}  of $\cF$ by $q_f(\cF)=\max\{p:\text{$\cF$ is weakly $p$-small}\}$.

It follows immediately from the definition that $q(\cF)\le q_f(\cF)$ (by taking the function $\gl$ to be the indicator function of $\mathcal G$).
The following conjecture of Talagrand
\cite[Conjecture 6.3]{talagrand2010many}
asserts a reverse inequality up to a universal constant.

\begin{conj}\label{LT}
    There is a universal $L>0$ such that for every finite set $V$ and increasing family $\cF \sub 2^V$,
    \[q_f(\mathcal F) \le Lq(\mathcal F).\]
\end{conj}
\nin 
Equivalently, 
\beq{eq:LT.equiv} \text{weakly $p$-small implies $(p/L)$-small.}\enq

Our main goal is to establish the following dimension-free comparison between $q$ and $q_f$.

\begin{thm}\label{thm:MTez}
 There is a universal $K >0$ such that for every finite set $V$ and increasing family $\cF \sub 2^V$,
 \[q_f(\cF) \le Kq(\cF)\cdot \max\{1,\log\log(1/q(\cF))\}.\]
\end{thm}

\nin We will in fact prove the following equivalent statement (up to changing the universal constant $K$):
\beq{eq:equiv}
\text{for every $p \in (0,1)$, if $\cF$ is weakly $p$-small, then it is $p/(K\cdot\max\{1,\log\log(1/p)\})$-small.}
\enq

\begin{proof}[Proof of the equivalence] Write $q, q_f$ for $q(\cF), q_f(\cF)$.
    That \eqref{eq:equiv} implies \Cref{thm:MTez} follows immediately from $q \le q_f$ and the fact that $f(t)= \max\{1,\log\log(1/t)\}$ is non-increasing.
    
    To see the converse: fix the universal constant $K \ge 1$ in \Cref{thm:MTez}, and suppose that $\cF$ is weakly $p$-small. Since \Cref{LT} holds when $q$ is bounded away from zero (after adjusting the constant $L$), we may assume that $q$ is sufficiently small, say, $q\le \min\{1/100, 1/K^3\}$. Thus \Cref{thm:MTez} gives that $p\le q_f \le Kq\log\log(1/q)$, and hence
    \[\log(1/p)\ge \log(1/q)-\log K-\log\log\log(1/q) \ge \frac{1}{2}\log(1/q).\]
    Therefore,
\[
p\le 2Kq\log\log(1/p)
 \le 2Kq\max\{1,\log\log(1/p)\}. \qedhere
\]
\end{proof}

To place \Cref{thm:MTez} in context, we first review some related background. 
\Cref{LT} has motivated a number of partial results. In particular, the conjectured constant-factor
comparison has been established under various additional assumptions
on the family $\cF$ or the function $\gl$
\cite{demarco2015note,frankston2022problem,fischer2023some,dubroff2024note,pham2025sharp,fischer2025further, fischer2025fractional}.

In the unrestricted setting, the best previously known general bound
is obtained by combining the results of Pham \cite[Theorem 1.2]{pham2025sharp} and Fischer
and Person \cite[Theorem 10]{fischer2023some}:
\beq{eq:best} q_f(\cF)\le Cq(\cF)\cdot \max\{1, \log\log|V|\}\footnote{Here we assume $|V| \ge 2$ to ensure that $\log\log|V|$ is meaningful.}.\enq

We remark that \Cref{thm:MTez} strengthens \eqref{eq:best}, up to a change in the universal constant. To see this, note that \eqref{eq:best} is equivalent to
\beq{eq:best'}\text{for every $p \in (0,1)$, weakly $p$-small implies $p/(C\cdot \max\{1,\log\log|V|\})$-small.}\enq
If $p \le 1/(2|V|)$,  taking $\cG$ to be the family of all singleton subsets of $V$ shows that every nontrivial increasing family is $p$-small. Therefore, even the conclusion of \eqref{eq:LT.equiv}  automatically holds  in this case. On the other hand, if $p>1/(2|V|)$, then 
\[
\max\{1,\log\log(1/p)\}
=
O\left(\max\{1,\log\log|V|\}\right).
\]
Thus, \eqref{eq:equiv} implies \eqref{eq:best'}.

The primary motivation for seeking a dimension-free comparison was the discrete version of Talagrand's Convexity Conjecture (see, e.g., \cite{talagrand2010many, ascoli2026reformulation}) -- see \Cref{rmk:dim}. In \Cref{sec:discussion}, we derive a  weaker version of the discrete Convexity Conjecture (\Cref{cor:DCC}), while retaining dimension independence.

Our derivation is inspired by the approach of \cite{dubroff2024note}. We replace their Lemma 1.3 (which is essentially \cite[Lemma~3.4]{bednorz2022suprema}) with \cite[Theorem 1.3]{pham2025sharp}. A similar argument appears in \cite[Section 3]{pham2025sharp}. %, although the present derivation was directly motivated by \cite{dubroff2024note}. 
We next restate \cite[Theorem 1.3]{pham2025sharp} in the form that we will use; this requires a little additional preparation. 

Let $\cF \sub 2^{V}$ be given. For each $I \in \cF$, let $\mu_I$ be a probability measure on $I$; that is, $\mu_I(v) \ge 0$ for every $v \in I$ and $\sum_{v \in I}\mu_I(v)=1$. For $A \sub I$, write $\mu_I(A):=\sum_{v \in A}\mu_I(v)$. We say that  $Y \sub V$ is \textit{$c$-bad} ($c >0$) if
\[\sup_{I \in \mathcal F} \mu_I(Y \cap I)<c.\]
For $r\in[0,1]$, let $V_r$ denote a random subset of $V$
obtained by including each element independently with probability $r$.
With this notation, \cite[Theorem 1.3]{pham2025sharp} can be stated as follows:

\begin{thm}\label{thm:Pham2}
    Let $q \in (0,1)$, and suppose $\cF$ is not $q$-small. If $s \ge 1$ is an integer and $16sq \le 1$, then
    \[\pr(\text{$V_{16sq}$  is $(1-2^{-s})$-bad})\le 2/3.\]
\end{thm}

\section{Proof of \Cref{thm:MTez}}

We prove \eqref{eq:equiv}. Our proof initially follows the argument of \cite{dubroff2024note}. Fix $p \in (0,1)$, and set
\[s=\left\lceil \log_2\log_2(32/p)\right\rceil+1, \quad q=p/(512s), \quad \rho=16sq.\]
For convenience, write $L$ for $2^s$.

Suppose that $\cF$ is weakly $p$-small, and let $\gl$ satisfy \eqref{eq:wps}. We may (and will) assume $\gl_\0=0$; indeed, if $\gl_\0>0$, then define $\gl'$ by 
\[\text{$\gl'_\0=0$ and 
$\gl'_S=\gl_S/(1-\gl_\0)$ if $S\neq \0$}.\] 
Then $\gl'$ also satisfies \eqref{eq:wps} (with 1/2 improved to 
$(1-2\gl_\0)/(2(1-\gl_\0))$).

For each $I \in \cF$, set 
\beq{eq:nu_def}\nu_I := \sum_{S \sub I} |S|\gl_S.\enq
Since $\gl_\0=0$ and $\sum_{S \sub I} \gl_S \ge 1$, we have $\nu_I \ge 1$.
Define $\mu_{I}:I \ra \RR^+$ by
\[\mu_{I}(v)=\nu_I^{-1}\sum_{v \in S \sub I} \gl_S  \,\,\,\, \forall v\in I\]
and notice that $\mu_I(I) = 1$. 

Suppose, for a contradiction, that $\cF$ is not $q$-small.   Then by \Cref{thm:Pham2} and the definition of $\rho$, the probability that $V_\rho$ is $(1-2^{-s})$-bad (with respect to the $\mu_I$'s) is at most $2/3$. We will obtain a contradiction by proving the following lower bound.

\begin{claim}\label{cl:contradiction} With probability at least $3/4$,
$V_{\rho}$ is $(1 - 2^{-s})$-bad with respect to the $\mu_I$'s.
\end{claim}

\begin{proof}

For any $Y \sub V$ and $I \in \cF$,
\[
\mu_{I}(Y \cap I) = \nu_I^{-1}\sum_{v \in Y \cap I}\sum_{v \in S \sub I}\gl_S =\nu_I^{-1}\sum_{S \sub I} |S \cap Y|\gl_S=1-\nu_I^{-1}\sum_{S \sub I} |S \setminus Y|\gl_S.
\]
Here we depart from the argument of \cite{dubroff2024note}. 
For $Y \sub V$, define\footnote{For readers familiar with the proof of \cite{dubroff2024note}: $\cB(Y)$ plays a role analogous to that of the family $\{S: S \sub Y\}$ in \cite{dubroff2024note}. The assumption on $\gl$ there allows us to restrict attention to sets $S$ with $|S|\le r$. For such sets, $S \not\sub Y$ implies $|S \setminus Y| \ge |S|/r$, a crucial point in that proof. Here, the analogous bound $|S \setminus Y|\ge 2|S|/L$ holds whenever $S \notin \cB(Y)$.}
\[\cB(Y)=\{S \sub V: |S \setminus Y|< 2|S|/L\}.\]
Since $S \notin \cB(Y)$ implies $L|S \setminus Y|/2\ge |S|$,
\[\begin{split}\nu_I^{-1}\sum_{S \sub I}|S \setminus Y|\gl_S&\ge \nu_I^{-1}\sum_{S \sub I, S \notin \cB(Y)} |S \setminus Y|\gl_S\\
&\ge \frac{2}{L\nu_I}\sum_{S \sub I, S \notin \cB(Y)}|S|\gl_S\\
&\stackrel{\eqref{eq:nu_def}}{=}\frac{2}{L}-\frac{2}{L\nu_I}\sum_{S \sub I, S \in \cB(Y)}|S|\gl_S.
\end{split}\]

\nin When $Y=V_\rho$, we may bound the last sum (for any $I$) by $\sum_{S \in \cB(V_\rho)}|S|\gl_S$, so it remains to prove the following bound.

\begin{claim}\label{cl:Bad}
    Write $X=V_\rho$. With probability at least $3/4$, $\sum_{S \in \cB(X)}|S|\gl_S <1/2$.
\end{claim}

\begin{subproof}
    Note that $S \in \cB(Y)$ iff $|S \cap Y|>(1-2/L)|S|$. Therefore, for $S \in {V \choose k}$,
    \[\pr(S \in \cB(X))\le \binom{k}{k-\lfloor 2k/L\rfloor}\rho^{k-\lfloor 2k/L \rfloor}\le2^k\rho^{k(1-2/L)}\le (4\rho)^k\]
    where the last inequality holds because, by our choice of parameters, $\rho^{-2/L}\le 2$.  
    Since $k \le 2^k$,
    \[k\cdot \pr(S \in \cB(X))\le (8\rho)^k.\]
    Thus
    \[\mathbb E\left[\sum_{S \in \cB(X)}|S|\gl_S\right]=\sum_{S\ne \emptyset}\pr(S \in \cB(X))|S|\gl_S\le \sum_{S \ne \emptyset}\gl_S(8\rho)^{|S|}=\sum_{S \ne \emptyset}\gl_S(p/4)^{|S|}\le \frac{1}{4}\sum_S \gl_Sp^{|S|}\le \frac{1}{8}.\]
    Now apply Markov's inequality.
\end{subproof}

By \Cref{cl:Bad} (and since $\nu_I\ge 1$ for any $I \in \cF$), with probability at least $3/4$,
\[\sup_{I \in \cF}\mu_I(V_\rho \cap I)<1-\frac{1}{L},\]
i.e., $V_\rho$ is $(1-1/L)$-bad.
\end{proof}

 \Cref{cl:contradiction} gives the desired contradiction, so  $\cF$ is $q$-small. Finally, since $
s=O\bigl(\max\{1,\log\log(1/p)\}\bigr)$ and $q=p/(512s)$, we obtain the conclusion in \eqref{eq:equiv}.

\section{Discussion}\label{sec:discussion}

We now discuss the connection with the discrete version of
Talagrand's Convexity Conjecture. For $p\in[0,1]$, let $\mu_p$
denote the distribution of $V_p$ on $2^V$; 
that is, for every $A \sub V$, $\mu_p(\{A\})=p^{|A|}(1-p)^{|V \setminus A|}$, and $\mu_p(\mathcal A)=\pr(V_p \in \mathcal A)=\sum_{A \in \mathcal A}\mu_p(\{A\})$ for every $\mathcal A \sub 2^V$.

For a family $\cA \sub 2^V$ and a positive integer $k$, set
\[\cA_{(k)}:=\{S \sub V: S \sub A_1 \cup \cdots \cup A_k \text{ for some } A_1, \ldots, A_k \in \cA\},\]
and $\cA^{(k)}:=2^V \setminus \cA_{(k)}$. Thus $\cA^{(k)}$ is the increasing family of sets not contained
in any union of $k$ members of $\cA$.

\begin{conj}[Talagrand \cite{talagrand1995all, talagrand2010many}]\label{conj:DCC}
    There exist a positive integer $k$ and a universal constant $L\ge 1$ such that, for every $p \in (0,1)$, finite set $V$, and  family $\cA \sub 2^V$, if $\mu_p(\cA) \ge 1-1/(2k)$ then $\cA^{(k)}$ is $(p/L)$-small.
\end{conj}

Hua, Song, and Tudose \cite{hua2026talagrand} resolved the original Gaussian version of Talagrand's Convexity Conjecture. Using an argument of Talagrand, they deduced the following partial result bound toward \Cref{conj:DCC}; see \cite[Corollary 1.4 and Appendix A]{hua2026talagrand}.

\begin{thm}\label{thm:HST}
    There exist $\eps>0$ and $L \ge 1$ so that for any finite set $V$, $p \in (0,1)$, and any family $\cA \sub 2^V$, if $\mu_p(\cA) \ge 1-\eps$ then $\cA^{(3)}$ is $p^L$-small.
\end{thm}

For comparison, a result of Pham and the author
\cite[Theorem 1.1]{park2024proof} implies that, for any $\cA \sub 2^V$ with
$\mu_p(\cA)\ge1/2$, the family $\cA^{(1)}$ is
$(p/(C\log\ell(\cA)))$-small for a universal constant $C$, where $\ell(\cA)$ is the maximum of 2 and the size of a largest minimal element of $\cA^{(1)}$.
Since $\cA^{(3)}\subseteq\cA^{(1)}$, this also gives the same
smallness bound for $\cA^{(3)}$. The two bounds from \cite{hua2026talagrand} and \cite{park2024proof} are  not uniformly comparable; we refer the readers to the discussion following
\cite[Corollary 1.4]{hua2026talagrand} for further discussion.

\begin{remark}\label{rmk:dim}
A key requirement of \Cref{conj:DCC} is that the smallness parameter depend only on the density parameter $p$, independently of the size of the ground set. The importance of dimension independence was emphasized by Talagrand in \cite[Section 1]{talagrand2010many}.
\end{remark}

Recently, Li \cite[Corollary 3.1]{li2026spread} proved a fractional version of \Cref{conj:DCC}. His result is formulated in terms of spread probability measures. By the duality relation between spread measures and being weakly-small \cite[Proposition 6.7]{talagrand2010many}, it yields the following statement in our terminology:

\begin{thm}\label{thm:Li}
        For every finite set $V$, $p \in (0,1)$, and family $\cA \sub 2^V$ satisfying $\mu_p(\cA)>1/2$, the family $\cA^{(2)}$ is weakly $(p/2)$-small.
\end{thm}

Combining \Cref{thm:MTez} with \Cref{thm:Li}, we obtain the following dimension-free bound.

\begin{cor}\label{cor:DCC}
There is a universal $L>0$ such that, for every finite set $V$, $p \in (0,1)$, and $\cA \sub 2^V$ satisfying $\mu_p(\cA)>1/2$, the family  $\cA^{(2)}$ is $p/(L\max\{1, \log\log(1/p)\})$-small.\footnote{By extending Li's argument to $\cA^{(2m)}$ for integers $m \ge 1$ and combining this with \Cref{thm:MTez}, one also obtains the conclusion that $\cA^{(2m)}$ is $\Omega(p/(L\max\{1,\log\log(1/p)\})^{1/m})$-small, where the implicit constant is universal.}
\end{cor}

\section*{Acknowledgments}
We thank Michel Talagrand for stimulating conversations and for his warm encouragement to publicize the findings in this manuscript. All original mathematical ideas presented in this work are the author's own, built upon prior joint work with collaborators. Generative AI was used during the revision stage, and the final text remains entirely the responsibility of the author. JP was supported by NSF CAREER Grant DMS-2443706  and a Sloan Fellowship.

\bibliographystyle{plain}
\bibliography{bibliography}

\end{document}